\documentclass{amsart}
\usepackage{amsmath}%
\usepackage{amsfonts}%
\usepackage{amssymb}%
\usepackage{graphicx}
\usepackage{fullpage}
\usepackage{color}
\usepackage[utf8]{inputenc}
\usepackage[english]{babel}
\usepackage{soul}
\usepackage{cancel}
\usepackage{float}

\newcommand{\la}{\lambda}
\newcommand{\Om}{\Omega}
\newcommand{\R}{\mathbb R}

\newcommand{\rd}{\color{red}}

\newtheorem{theorem}{Theorem}[section]

\newtheorem{corollary}[theorem]{Corollary}

\newtheorem{definition}[theorem]{Definition}

\newtheorem{lemma}[theorem]{Lemma}

\newtheorem{proposition}[theorem]{Proposition}
\newtheorem{remark}[theorem]{Remark}

\title{Pleijel's theorem for Schrödinger operators}
\author{Philippe Charron}
\address{Philippe Charron\newline Mathematics Department
	Technion – Israel Institute of Technology, 	Haifa 32000,	Israel\newline
	Department of Mathematics, University College London, Gower Street, LONDON, WC1E  6BT, UK}
\email{p.charron@ucl.ac.uk}

\author{Corentin Léna}
\address{Corentin Léna\newline Institut de Mathématiques, Bâtiment UniMail,
Rue Emile-Argand 11, 2000 Neuchâtel, Switzerland\newline
Universit\`a degli Studi di Padova, Dipartimento di Tecnica e Gestione dei Sistemi Industriali (DTG), Stradella S. Nicola 3, 36100 Vicenza, Italy\newline
Universit\`a degli Studi di Padova,  Dipartimento di Matematica ``Tullio Levi-Civita", via Trieste 63, 35121 Padova, Italy.}  \email{corentin.lena@unipd.it}

\date{January 24, 2025}

\begin{document}

\begin{abstract} {We are concerned in this paper with the real eigenfunctions of Schrödinger operators. We prove an asymptotic upper bound for the number of their nodal domains, which implies in particular that the inequality stated in Courant's theorem is strict, except for finitely many eigenvalues. Results of this type originated in 1956 with Pleijel's Theorem on the Dirichlet Laplacian and were obtained for some classes of Schrödinger operators by the first author, alone and in collaboration with B. Helffer and T. Hoffmann-Ostenhof. Using methods in part inspired by work of the second author on Neumann and Robin Laplacians, we greatly extend the scope of these previous results. }
\end{abstract}

\maketitle

Contribution statement: The authors contributed to the works equally.

\newpage

\section{Introduction}

\subsection{Setting}

 The aim of this paper is to give a generalization of Pleijel's theorem for the Dirichlet Laplacian to a large class of Schrödinger operators $H_V = -\Delta + V$ in $\mathbb{R}^d$. We want to give upper bounds on the number of nodal domains of an eigenfunction $f_\la$ of $H_V$. A nodal domain of a function $f: \Om \to \mathbb{R}$ is defined as a connected component of $\Om \backslash f^{-1}(0)$. Throughout this paper, the number of nodal domains of $f$ will be denoted by $\mu(f)$.

\subsection{Past results}

Over the years, numerous upper bounds have been found for the number of nodal domains of Laplace and Schrödinger eigenfunctions. 

\bigskip

Given $\Om \subset \mathbb{R}^d$  open, bounded and connected, let $\{\la_n\}$, $n \geq1$, be the eigenvalues of the Dirichlet Laplacian on $\Om$, numbered in  increasing order and with account of multiplicity, and let $ \{f_n\}$ be a basis of real-valued eigenfunctions associated to $\{\la_n\}$. 

\bigskip

The first upper bound on $\mu(f_n)$ was given by Courant in 1923 \cite{Cou1923}: for any $n \geq 1$ and any eigenfunction $f_n$ associated with $\la_n$, $\mu(f_n) \leq n$. Using a different method, Pleijel showed in 1956 \cite{Ple56} that the equality in Courant's theorem can only be attained finitely many times. Furthermore, he gave the following stronger asymptotic result:
\begin{equation}
	\label{eqPleijel}
	\limsup\limits_{n\to \infty} \frac{\mu(f_n)}{n} \leq \gamma_d \, ,
\end{equation}
where $\gamma_d$ is a constant that depends only on the dimension of $\Om$. Precisely, $\gamma_d$ is given by the following formula:
\begin{equation}
	\gamma_d = \frac{2^{2d-2} d^2 \Gamma(\frac{d}{2})}{j_{\frac d2 - 1}^d} \, .
\end{equation}
Here, $j_a$ denotes the first positive zero of the Bessel function of the first kind $J_a$. In particular, $\gamma_d < 1$ for all $d \geq 2$ { (see \cite{BerMey1982,HelPer2016})}.

\bigskip

Pleijel's result was generalized to closed manifolds and manifolds with boundary and Dirichlet boundary conditions in 1957 \cite{Pee1957} and in 1982 \cite{BerMey1982}. More recently, Pleijel's upper bound was proven for domains $\Om \subset \mathbb{R}^2$ with piecewise-analytic boundary and Neumann boundary conditions in 2009 \cite{Pol2009} and then for domains in $\mathbb{R}^d$ with $C^{1,1}$ boundary with either Neumann or some types of Robin boundary conditions  in 2019 \cite{Len2019}.

\bigskip

In the case of Schrödinger operators, Pleijel's asymptotic upper bound was found to hold for the quantum harmonic oscillator in 2018 \cite{Cha2018} as well as for Schrödinger operators with radial potentials in \cite{ChaHelHof2018}, which included both positive and negative potentials, possibly having a singularity at the origin.

\bigskip
\subsection{Aim of this paper}

Our goal is to prove estimate \eqref{eqPleijel} for Schrödinger operators with non-radial potentials. The methods used in \cite{Cha2018} and \cite{ChaHelHof2018} were highly dependant on the potential being radial to estimate the number of nodal domains. However, these methods are quite rigid and are not easily transfered to the non-radial case. In this paper, we introduce new tools that allow us to treat either radial or non-radial potentials. Like in \cite{ChaHelHof2018}, we consider both positive and negative potentials.

\bigskip

For example, one classical example that was not covered by previous results is a linear combination (with positive coefficients) of Coulomb potentials, which correspond to an electron in the electric field generated by a finite number of nuclei in free space. Our estimates also hold for perturbations of such potentials.

\bigskip 

\subsection{Main results}

 In this paper, we consider operators of the form $H_V = -\Delta + V$ in $\mathbb{R}^d$, with $V$ a real-valued function. These operators preserve real-valued functions, and accordingly we will only consider real-valued eigenfunctions in the rest of the paper. We will work with two different cases which will be treated differently.
 
 \bigskip
\begin{definition}\label{casea}
	Case A: $V\in C^1(\mathbb{R}^d)$ and there exist $C_1,C_2,C_3 >0$ and $a,b,c >0$, with $a\le b$, such that 
	
	\begin{align}
		C_1 |x|^a \leq V(x) \leq C_2 |x|^b \, , \\
		|\nabla V(x)| \leq C_3|x|^c \, .
	\end{align}
	
\end{definition}

Note that we will write throughout \emph{$f(x)\sim g(x)$ as $x\to a$} when
	$\lim_{x\to a}f(x)/g(x)\to 1$, where $f$ and $g$ are functions defined in a neighborhood of $a$, with $g>0$. {\ Also,} we will write $f(x) \asymp g(x)$ if there exists $ C\ge 1$ such that $ C^{-1}\le f(x)/g(x)\le C$ in a neighborhood of $a$.

\begin{definition}\label{caseb}
	Case B: there exist $N$ points $x_1, x_2 \ldots x_N\in\mathbb{R}^d$ (called \emph{ singular points}), with $N\ge0$,  such that $V\in C^1(\mathbb{R}^d\setminus\{x_1\dots,x_N\})$ and  $ V(x) \sim - c_{i}|x-x_i|^{ -a_i}$ when $x\to x_i$, with $0 < a_i <2$ and $c_i > 0$ for $1 \leq i \leq N$. Also, there exist  $C_1,C_2,C_3 >0$, $0 < a \leq b < 2$ and $c>0$ such that
	 
	\begin{align}
		-C_1 |x|^{-a} \leq V(x) \leq -C_2 |x|^{-b} \, , \\
		|\nabla V(x)| \leq C_3 |x|^{-c} \, ,
	\end{align}
	
\noindent when $|x|> M$, for a large enough $M$.
\end{definition}

\bigskip

\begin{remark}	\label{remc}
	It is easy to see that the exponents in definition \ref{caseb} must satisfy $c\leq b+1$, in particular $c<3$.
\end{remark}

\begin{remark}
	We have not given the relevant sufficient conditions on the exponents $a$, $b$ and $c$ yet. They will be specified in theorems \ref{thma} and \ref{thmb}.
\end{remark}

\begin{remark}
	We point out that case A holds for the sum of  finitely many harmonic oscillator potentials $(a=b=2$, $c=1)$ and case B for atomic Hamiltonians (that is to say when $V$ is the sum of finitely many Coulomb potentials, $a=b=1$, $c=2$).
\end{remark}

\bigskip

In both cases, the quadratic form associated to the differential operator $-\Delta+V$, acting on $C^\infty_c(\R^d\setminus X)$ {(where $X=\emptyset$ in case A and $X=\{x_1,\dots,x_N\}$ in case B)}, is lower-semi bounded, and we can therefore define the self-adjoint operator $H_V$ as the Friedrichs extension of this differential operator, as we show in section \ref{sechardy} and appendix \ref{annex0}. We restrict our attention to $H_V$ in the rest of the paper, although the differential operator might have other self-adjoint extensions in some cases. The following result is also proved in section \ref{sechardy} and appendix \ref{annex0}.

\begin{proposition} \label{propEigenvalues}In case A, $\sigma(H_V)=\{\la_n\}_{n\ge1}$, with each $\la_n$ an eigenvalue of finite multiplicity and $\la_n\to+\infty$. 

In case B, 
\[\sigma_{ess}(H_V)=[0,+\infty)\]
and
\[\sigma(H_V)\cap(-\infty,0)=\{\la_n\}_{n\ge1} \, ,\]
with each $\la_n$ an eigenvalue of finite multiplicity and $\la_n\to 0$. 

In both cases, the eigenvalues are numbered in increasing order with account of multiplicity. 
\end{proposition}

\bigskip

The first main result of this paper is the following:

\begin{theorem}[Case A]\label{thma}
	Let $H_V = -\Delta +V$ with $V \in C^1(\mathbb{R}^d)$ as in Case A, see Definition \ref{casea}, and $a,b,c$ satisfying the following:
	
	\begin{align}\label{caseAcond}
		\frac{d}{a} + \frac{c}{3a} - \frac{1}{2} <  \frac{d}{b} \, .
	\end{align}
	
	Let the sequence $\{\lambda_n\}$, with $n\ge1$, consist of the eigenvalues of $H_V$, numbered in increasing order with account of multiplicity, and, for all $n\ge1$, let $f_n$ be any eigenfunction associated to $\lambda_n$. Then, we have the following estimate:
	
	\begin{equation}\label{Pleipos}
		\limsup\limits_{n \to \infty}\frac{\mu(f_n)}{n} \leq \gamma_d\, .
	\end{equation}
\end{theorem}

Now, these implicit conditions on $a,b$ and $c$ are hard to parse, but allow us to consider potentials where $a \neq b$. If $a=b$ we get a simpler condition on $c$:

\bigskip

\begin{corollary}
	Let $V$ be as in case A with $a=b>0$ and $c < \frac{3}{2}a$. Then, estimate (\ref{Pleipos})  holds.
\end{corollary}

The second main result of this paper is the following:

\begin{theorem}[Case B]\label{thmb}
	Let $H_V = -\Delta +V$ with $V$ as in Case B, see Definition \ref{caseb}, and $a,b,c$ satisfying the following:
	
	\begin{align}\label{caseBcond}
		- \frac{d}{a} + \frac{c}{3a} - \frac{1}{2} >  -\frac{d}{b} \, .
	\end{align}

	Let the sequence $\{\lambda_n\}$, with $n\ge1$, consist of the eigenvalues of $H_V$ below the essential spectrum, numbered in increasing order with account of multiplicity and, for all $n\ge1$, let $f_n$ be any eigenfunction associated to $\lambda_n$. Then, we have the following estimate:
	
	\begin{equation}\label{Pleineg}
		\limsup\limits_{n \to \infty}\frac{\mu(f_n)}{n} \leq \gamma_d \,.
	\end{equation}
\end{theorem}

\bigskip

Again, if $a=b$, we have the simpler condition on $c$:

\bigskip

\begin{corollary}
	Let $V$ be in case B with $a=b$ and $ \frac32 a<c $. Then, estimate (\ref{Pleineg}) holds.
\end{corollary}

\bigskip

\subsection{Constraints and conditions}
The conditions on $a,b$ and $c$ defined in equations (\ref{caseAcond}) and (\ref{caseBcond}) appear in our calculations when we try to control the error terms. It is not clear if those conditions are sharp.

\bigskip

Also, we do not study the case where  there are singularities near the origin and the potential grows at infinity. We believe that our method could treat such potentials as well.

\bigskip

We also believe that our method could be generalized to manifolds, as long as a suitable version {of the Faber-Krahn} inequality and Weyl's law can be found. Indeed, as seen in \cite{Roz1974}, the conditions on $V$ given by definitions \ref{casea} and \ref{caseb} are far from being the most general under which Weyl's law holds true. Furthermore, there exists classes of potential $V$ such that $H_V$ has a discrete spectrum whose asymptotic distribution does not follow Weyl's law (see \cite{Rob1982,Sim1983} and references therein). It would be worth investigating upper bounds analogous to \eqref{eqPleijel} in those cases.

\bigskip

\subsection{Outline of the proof}

The current paper expands on the methods from \cite{Cha2018}, \cite{ChaHelHof2018} and \cite{Len2019}. In \cite{Cha2018} and \cite{ChaHelHof2018}, the two main ideas were to first partition the potentials into layers to optimally control the size of nodal domains inside each layer, and use algebraic geometry to limit the number of nodal domains that cross from one layer to another. Choosing a suitable partition to balance both estimates gave the final result. However, this limited the scope of the proof to potentials which had a basis of eigenfunctions with a polynomial behaviour. It also used the fact that level sets of the potential were hyperspheres in order to apply the algebraic geometry tools. Therefore, we need to introduce new tools to study non-radial potentials.

\bigskip

In \cite{Len2019}, one of the key ideas was to separate nodal domains depending on where the $L^2$-mass of the eigenfunction was distributed, and then to show that the number of nodal domains which have a high $L^2$-mass close to the boundary was negligible. 

\bigskip

In this paper, we combine the two methods in a novel way. Let $\Om$ be a nodal domain of some eigenfunction $f_\la$. We create a locally finite partition of unity $\{A_i^2\}$ of $\mathbb{R}^d$, then we find at least one $i$ such that $||f_\la||_{L^2\left(\Om \cap \text{supp}(A_i)\right)}$ is comparable to $||f_\la A_i||_{L^2\left(\Om \cap \text{supp}(A_i)\right)}$. This will be done in sections \ref{conspos} and \ref{consneg}. It will allow us to use {the Faber-Krahn} inequality in an optimal way in each element of the partition. This will be made more precise in section \ref{secPrelim}. By applying {the Faber-Krahn} inequality inside each region, we bound the total number of nodal domains by a main term which can be computed in terms of the Weyl asymptotics and an error term which can be bounded with suitable assumptions on $V$. This will be done in sections \ref{secpos} and \ref{secneg}.

\bigskip

At this point, all that is left to do is to compare the nodal domain estimates to the eigenvalue counting function. For the sake of completeness, we made in both cases an explicit calculation for a lower bound of the counting function, using Dirichlet bracketing. This is done in sections \ref{weylpos} and \ref{weylneg}.

\bigskip

\subsection{Acknowledgments}

 The authors would like to thank Bernard Helffer for many stimulating discussions around this topic, as well as Rami Band and Iosif Polterovich for the helpful comments and suggestions.
 
The research for this article was done in part for the PhD thesis of P. Charron at the Université de Montréal, which was submitted while this work was still in progress. As such, he would like to thank Iosif Polterovich for his guidance during the project.

C. Léna acknowledges the support of COST (European Cooperation in Science and Technology) through the COST Action CA18232 MAT-DYN-NET.

\bigskip

\section{Preliminary results}

\subsection{Lower semiboundedness and the Friedrichs extension} \label{sechardy}

Let us first give more details on the definition of our self-adjoint operator sketched in the introduction. As in Definition \ref{caseb}, we assume that we are given a finite set of points $X=\{x_1,\dots,x_N\}$ in $\mathbb R^d$ (called \emph{singular points} or \emph{poles}). We define the self-adjoint operator $H_V=-\Delta+V$ as the Friedrichs extension of the sesquilinear form $a(\cdot,\cdot)$ given by
\begin{equation*}	
	a(u,v):=\int_{\R^d}\nabla u\cdot\overline{\nabla v}+\int_{\R^d}V(x)u\overline{v}
\end{equation*} 
for $(u,v)\in C^\infty_c(\R^d\setminus X)\times C^\infty_c(\R^d\setminus X)$, {where $C^\infty_c(\R^d\setminus X)$ is the space of smooth compacyly supported function on $\R^d$}. This standard method of definition works as soon as $a(\cdot,\cdot)$ is lower semi-bounded, as seen from \cite[Th. X.23]{ReeSim2}.  In particular, $a(\cdot,\cdot)$ is closable in this case. Therefore, the definition of $H_V$ ultimately relies on the following result, proved in appendix \ref{annex0}.

\bigskip

\begin{lemma}\label{lemLSB} There exists a constant $C$, depending only on $d$ and $V$, such that, for all $u\in C^\infty_c(\R^d\setminus X)$, 
\begin{equation}\label{eqLSB}
	a(u,u)\ge -C\|u\|_{L^2}^2\,.
\end{equation}
\end{lemma}

\bigskip

We can now describe the spectrum of the operator $H_V$ according to the nature of the potential $V$ and give some properties of its eigenfunctions. In what follows, we define the form domain $\mathcal Q_V$ of $H_V$ as the completion of $C^\infty_c(\R^d\setminus X)$ for the Hilbert norm
\begin{equation*}
\|u\|_1^2:=\langle u,H_Vu\rangle+(C+1)\|u\|^2 \, ,
\end{equation*}
where $u\in C^\infty_c(\R^d\setminus X)$ and $C$ is the constant given by Lemma \ref{lemLSB}. We summarize our analysis in the following proposition, proved in appendix \ref{annex0}.

\begin{proposition} \label{propSpectrum} 
In case A, the spectrum of $H$ consists of eigenvalues of finite multiplicity tending to $+\infty$ and
\begin{equation*}
	\mathcal Q_V=\left\{u\in H^1(\R^d)\,;\,|V|^\frac12u\in L^2(\R^d)\right\}\,.
\end{equation*}

 In case B,	$\sigma_{\mbox{ess}}(H_V)=[0,+\infty)$ and $Q_V=H^1(\R^d)$.
\end{proposition}

\begin{remark} In case B, the negative part of $\sigma(H_V)$ consists of a sequence of eigenvalues with finite multiplicities. This sequence could a priori be finite or even empty. However, the estimate of Lemma \ref{lemweylneg1} shows in particular {the existence of infinitely many negative eigenvalues, so that we can write}
\begin{equation*}
	\sigma(H_V)\cap(-\infty,0)=\left\{\lambda_n; n\ge 1\right\},
\end{equation*}
where the $\lambda_n$ are repeated according to their multiplicity and $\lambda_n \to0$ as $n\to\infty$. This concludes the proof of Proposition \ref{propEigenvalues} in the introduction.
\end{remark}

\begin{remark}\label{remLSB} From the definition of the form domain, inequality \eqref{eqLSB} holds for all $u\in Q_V$, with the same constant $C$.
\end{remark}

\subsection{Partition and Rayleigh quotient estimates} \label{secPrelim}

We gather in this section several lemmas which will be used in the rest of the paper. Our proof relies, as in Pleijel's original paper, on the use of {the Faber-Krahn} inequality to have lower bounds on the volume of each nodal domain. However, in the case of Schrödinger operators, the Rayleigh quotient over each nodal domains depends on $V$ as well as $\la$ and using Faber-Krahn directly will not give a good enough upper bound in the final estimate.

\bigskip

In order to get better bounds, we will show that for every nodal domain $\Om$ of $f_\la$, it is possible to localize $f_\la$ with a partition of unity $\left\{ A_i^2\right\}$ and use {the Faber-Krahn} inequality in its refined form inside the support of at least one $A_i$. This will show that the volume of $\Om \cap \text{supp}A_i$ is bounded from below.

\bigskip

However, by using this method, the Rayleigh quotient of $A_i f_\la$ will be affected by the norms of $|\nabla A_i|^2$ and by the size of the overlap in the partition. We will need to choose both carefully in order to get the best estimates.

\bigskip

As the partitions that we will be using in our proof may depend on both $\la$ and $V$, we shall prove general results for partitions under suitable conditions, and make sure that these conditions are met in our constructions.

\bigskip

Let us denote by $\{A_i^2\}_{i\in I}$ a smooth partition of unity (meaning that $A_i\in C^\infty_c(\R^d)$ and $\sum A_i^2=1$) and by $\{U_i\}_{i\in I}$ an open covering of $\R^d$ such that $\mbox{supp}(A_i)\subset U_i$. We assume that the index set $I$ is countable and the covering $\{U_i\}$ is locally finite. 

\bigskip

Let us now consider an eigenvalue $\la$ of $H$, $f_\la$ an associated real-valued eigenfunction and $\Om$ one of its nodal domain. For all $i\in I$, we write $\Omega_i:=\Om\cap U_i$, see figure \ref{figLoc}.

\begin{figure}[H]
	
	\centering
	\includegraphics[width=7cm, height=6cm]{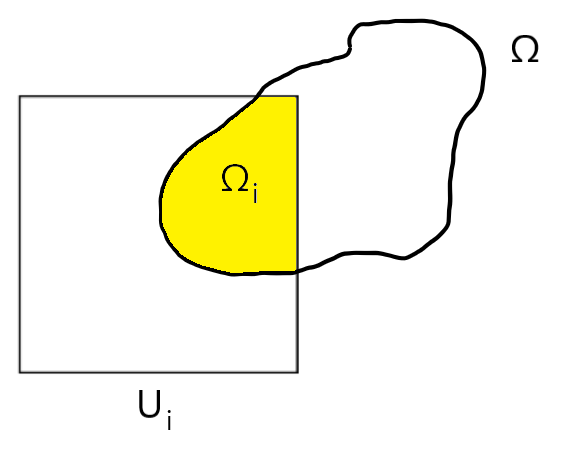}
\caption{Localization of a nodal domain.}
\label{figLoc} 
\end{figure}
\bigskip

\bigskip

\begin{definition}\label{defCLoc} For any $M>0$, we  say the nodal domain $\Omega$ of $f_\la$ is  \emph{ $M$-localized in $U_i$} when $\Om_i \neq \emptyset$ and
\begin{equation}
	\label{eqCLoc}
	\frac{\int_{\Om_i}f_\la^2}{\int_{\Om_i}A_i^2f_\la^2}\le M \,.
\end{equation}
\end{definition}

The next lemma tells us how to choose  $M$, for a given partition, so that that any nodal domain is $M$-localized in at least one $U_i$.

\begin{lemma} \label{lemCovLoc}Suppose that the \emph{multiplicity} of the covering $\{U_i\}$ is at most $M$, that is to say every point of $\R^d$ is contained in at most $M$ distinct $U_i$'s. Then, for each nodal domain $\Om$ of $f_\la$, there exists $i_0$ such that 
\begin{equation*}
	\frac{\int_{\Om_{i_0}}f_\la^2}{\int_{\Om_{i_0}}A_{i_0}^2f_\la^2}\le M \, .
\end{equation*}
\end{lemma}

\begin{proof} If it is not the case, then
\begin{equation*}
\int_{\Om}{\bf 1}_{U_i}f_\la^2=\int_{\Om_i}f_\la^2>M\int_{\Om_i}A_i^2f_\la^2
\end{equation*}
for all $i$ such that $\int_{\Om_i}A_i^2f_\la^2>0$ (${\bf1}_{U_i}$ is the characteristic function of the set $U_i$). Summing over all these $i$'s we find
\begin{equation*}
\int_{\Om}\left(\sum {\bf1}_{U_i}\right)f_\la^2>M\int_{\Om}\left(\sum A_i^2\right)f_\la^2=M\int_{\Om}f_\la^2\,.
\end{equation*}
 According to the hypothesis, $\sum{\bf1}_{U_i}\le M$ pointwise and we have a contradiction.
\end{proof}

\begin{remark}
Note that in the rest of the paper, the partition $\{A_i^2\}$ and the corresponding open covering $\{U_i\}$ may depend on $\la$ and $V$. However, the multiplicity of $\{U_i\}$ will be bounded throughout by a constant $M$ depending only on the dimension.
\end{remark}

As in Pleijel's paper \cite{Ple56}, we use the Faber-Krahn inequality in the following form. If $\Om\subset\R^d$ and $\rd f$ is a function which is smooth in $\Omega$ and vanishes on $\partial \Omega$, we define the Rayleigh quotient
\begin{equation*}
	R_\Om(f):=\frac{\int_\Om|\nabla f|^2}{\int_\Om f^2}=\frac{\langle -\Delta f,f\rangle_{L^2(\Om)}}{\|f\|_{L^2(\Om)}^2} \, .
\end{equation*}
Then, we have
\begin{equation}\label{eqFK}
	R_\Om(f)\ge \lambda_1^D(\Om)\ge K_d|\Om|^{-\frac2d} \, ,
\end{equation}
where $K_d$ is a constant depending only on the dimension $d$. Explicitly,
\begin{equation}\label{constFK}
	K_d=\lambda_1^D(\mathbb B^d)w_d^{\frac2d} \, ,
\end{equation}
where $\mathbb{B}^d$ denotes the unit ball in $\R^d$ and $w_d$ its volume. In the following, we will apply this inequality when $f $ is an eigenfunction of $H_V$ multiplied by an element of the partition of unity. This will  yield lower bounds for the volume of a nodal domain and ultimately an upper bound of the nodal count.

\bigskip

The next lemma will help us to derive estimates for the Rayleigh quotient of $A_i f_\la$.

\begin{lemma} We have, for all $i$,
	\begin{equation}
		\label{eqIBP}
		\int_{\Om_i} \left|\nabla (A_if_\la)\right|^2=\int_{\Om_i}(\la-V(x))\left(A_if_\la\right)^2+\int_{\Om_i}\left|\nabla A_i\right|^2f_\la^2\,.
	\end{equation}
\end{lemma}

\begin{proof} Expanding the gradient of a product on the left,
	\begin{equation*}
		\int_{\Om_i} \left|\nabla (A_if_\la)\right|^2=2\int_{\Om_i} A_if_\la\nabla A_i\cdot\nabla f_\la+\int_{\Om_i} A_i^2\left|\nabla f_\la\right|^2+\int_{\Om_i}\left|\nabla A_i\right|^2f_\la^2\,.
	\end{equation*}
	Using the identity $\nabla(A_i^2)= 2A_i \nabla(A_i)$, we find
	\begin{equation*}
			2\int_{\Om_i} A_if_\la\nabla A_i\cdot\nabla f_\la=\int_{\Om_i} \nabla \left(A_i^2\right)\cdot f_\la\nabla f_\la\,.
	\end{equation*}
	Using Green's identity and recalling that $ A_i^2f_\la=0$ on $\partial \Omega_i$,  we find
	\begin{equation*}
	\int_{\Om_i} \nabla \left(A_i^2\right)\cdot f_\la\nabla f_\la=-\int_{\Om_i} A_i^2\left|\nabla f_\la\right|^2-\int_{\Om_i} A_i^2f_\la\Delta f_\la\,.
	\end{equation*}
	Since $-\Delta f_\la=(\la-V(x))f_\la$, the result follows.
\end{proof}

 Identity \eqref{eqIBP} is actually the so-called \emph{Localization Formula} or \emph{IMS  Formula}. We refer the reader to \cite[Th. 3.2]{Cyc1987} (and reference therein) as well as \cite[Th. 1.1]{HelSjo1984} in a semi-classical context.

\begin{lemma}\label{lemRQ} Let $M>0$ and let us assume that $\Om$ is $M$-localized in $U_i$. Then, 
\begin{equation}\label{eqRQ}
	R_{\Om_i}(A_if_\la)\le \sup_{U_i}(\la-V +M|\nabla A_i|^2) \, .
\end{equation}
\end{lemma}

\begin{proof} Dividing identity \eqref{eqIBP} by $\int_{\Om_i} A_i^2f_\la^2$ and taking the supremum under the integral in the numerators gives us
\[R_{\Om_i}(A_if_\la)\le\sup_{\Om_i}(\la-V)+\sup_{\Om_i}|\nabla A_i|^2\,\frac{\int_{\Om_i}f_\la^2}{\int_{\Om_i}A_i^2f_\la^2}\,.\]
Using inequality \eqref{eqCLoc} and the inclusion $\Om_i\subset U_i$, we obtain inequality \eqref{eqRQ}.
\end{proof}

\begin{lemma}\label{lemCounting} For given $A_i$, $U_i$ and $M>0$, we denote by $\mu_i$ the number of nodal domains which are $M$-localized in $U_i$. Then,
\begin{equation*}
	\mu_i\le K_d^{-\frac{d}|}|U_i| \sup_{U_i}(\la-V +M|\nabla A_i|^2)^\frac{d}2 \, .
\end{equation*}
\end{lemma}

\begin{proof} For such a domain $\Omega$, it follows from inequalities \eqref{eqFK} and \eqref{eqRQ} that 
	\begin{equation*}
		1\le K_d^{-\frac{d}2}|\Omega_i|R_{\Omega_i}(A_if_\la)^\frac{d}2\le K_d^{-\frac{d}2}|\Omega_i|\sup_{U_i}(\la-V +M|\nabla A_i|^2)^\frac{d}2 \, .
	\end{equation*}
	We obtain the result by summing over all such $\Omega$, recalling that the corresponding $\Om_i = \Om \cap U_i$ are contained in $U_i$ and mutually disjoint.
\end{proof}

\begin{remark}
	It is possible for a nodal domain to be $M$-localized in more than one $U_i$. Our method of counting is not optimal in that it counts these {domains} multiple times. However, we do not know at this moment if any improvement could be obtained by removing the nodal domains that have been counted more than once.
\end{remark}

\section{Potentials growing at infinity} 

\subsection{Construction of the partition} \label{conspos}

\bigskip

Take the following partition: let $z=(z_1, \ldots, z_d) \in \mathbb{Z}^d$ and $\delta > 0$ small. Set $\rho = \la^{-m}$, with $m$ to be determined later.

\bigskip

Let $J_{z,\delta, \rho}$ be the hypercube $\rho(z_1-\frac{1}{2} -\delta, z_1+\frac{1}{2}+\delta) \times \rho(z_2 - \frac{1}{2} -\delta, z_2+\frac{1}{2}+\delta) \ldots \rho(z_d - \frac{1}{2} - \delta, z_d + \frac{1}{2} +\delta)$  and
$J_{z,\rho}$ the smaller hypercube $\rho(z_1-\frac{1}{2} , z_1+\frac{1}{2}) \times \rho(z_2 - \frac{1}{2} , z_2+\frac{1}{2}) \ldots \rho(z_d - \frac{1}{2} , z_d + \frac{1}{2} )$.

\bigskip

We know that for any  $ 0<\delta\le \frac{1}{2}$ and $\rho >0$, the collection $\left\{ J_{z,\delta,\rho} \, ,\, z \in \mathbb{Z}^d\right\}$ covers $\mathbb{R}^d$ and each $x \in \mathbb{R}^d$ is covered by at most $2^d$ hypercubes, see figure \ref{figCov}. We set $M:=2^d$ throughout.

\bigskip

Let $\ \{ A_{z,\delta,\rho}^2\}$ be a partition of unity subordinated to the cover $J_{z,\delta,\rho}$ {and such that $ A_{z,\delta,\rho}=1$ on $J_{z,\rho}$}. For each $\delta$, { we can construct it so that} there exists a constant $C(d,\delta)$ such that $\sup\limits_{J_{z,\delta,\rho}}|\nabla  A_{z,\delta,\rho}| \leq C(d,\delta)\rho^{-1}$.

\subsection{Nodal count estimates}\label{secpos}

{Note that in the rest of this section, $a$, $b$, $c$, $C_1$, $C_2$, $C_3$ denote the constants appearing in definition \ref{casea}.}
 We know that for any $\delta >0$, $\rho >0$ and a given nodal domain $\Om$ of $f_\la$, there exists at least one $z \in \mathbb{Z}^d$ such that $\Om$ is $M$-localized in $J_{z,\delta,\rho}$.

\bigskip

We set $\Om_z := \Om \cap J_{z,\delta,\rho}$.

\begin{figure}[H]
	\centering
	\includegraphics[width=7cm, height=6cm]{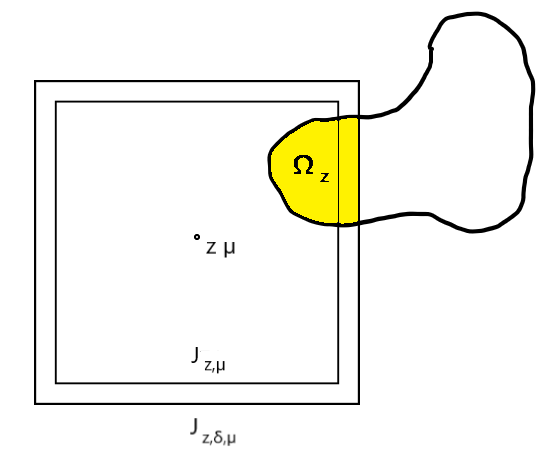}
	\caption{A covering hypercube.}
	\label{figCov}
\end{figure}
\bigskip

\bigskip

Let $B_z$ be the set of nodal domains $\Omega$ which are $M$-localized in $J_{z,\delta,\rho}$.

\bigskip

Since any nodal domain is contained in at least one set $B_z$, using lemma \ref{lemCounting}, recalling that $|J_{z,\delta,\rho}|=(1+2\delta)^d \la^{-md}$ and summing over all hypercubes gives us the following:

\begin{align}
\mu(f_\la) \leq \sum_{z \in \mathbb{Z}^d} (1+2\delta)^d \la^{-md} {K_d^{-\frac{d}{2}}} \sup_{J_{z,\delta,\rho}} \left[ \la - V + C\la^{2m} \right]_+^{\frac{d}{2}} \, ,
\end{align}

with $K_d$ defined in equation \eqref{constFK}. 

\bigskip
{
Now, let 
\[M_\la := \sum_{z \in \mathbb{Z}^d} \la^{-md}  \sup_{J_{z,\delta,\rho}} \left[ \la - V + C\la^{2m} \right]_+^{\frac{d}{2}} \, \]
and
\[W_\la := \int_{\mathbb{R}^d}(\la-V)_+^{\frac{d}{2}} \, .\]

We want to show that 

\begin{align}\label{mlawla}
	\mu(f_\la) \leq (1+2\delta)^d  K_d^{-\frac{d}2} W_\la (1 + o(1))\,,
\end{align}

\noindent when $\la \to \infty$ for a fixed $V$.

\begin{lemma}\label{Estdompos}
	We assume $c <\frac32a$. Then, there exists a constant $C(d,V,\delta)$ (depending only on $d$, $V$ and $\delta$), such that
	
	\begin{align}
		\mu(f_\la) \leq (1+2\delta)^d  K_d^{-\frac{d}{2}} W_\la	+ C(d,V,\delta) \la^{ \frac{d}{2} + \frac{d}{a} + \frac{ 2c}{3a} - 1     }\,.
	\end{align}	
\end{lemma}

The proof of this lemma will be done in appendix \ref{annex1}.

\bigskip

Now, we estimate the size of $W_\la$ from below:

\begin{align} \label{Wlapos}
	W_\la \geq \int\limits_{B(0, (\la/C_2)^{1/b})}\left[ \la - C_2 { |x|}^b \right]^{\frac{d}{2}} {\,dx}
	=C(d,V) \la^{\frac{d}{2}+ \frac{d}{b}}\,.
\end{align}

\bigskip

\subsection{Weyl's law for positive exponents}\label{weylpos}

For our purposes, we only need lower bounds for
\[ N(\lambda) := \sharp\left\{ \lambda_n < \lambda \right\}\,.\]

\bigskip

\begin{lemma}\label{Weylpos}
	Let us assume $c \leq \frac32 a$. Then, we have
	\begin{align}
		N(\la) \geq (2\pi)^{-d}w_d W_\la + O( \la^{\frac{d}{2} + \frac{d}{a} + \frac{c}{3a} - \frac{1}{2}     })\,.
	\end{align}
\end{lemma}

The proof of this lemma will be given in appendix \ref{annex1} using a Dirichlet bracketing argument, based on the partition $\{J_{z,\rho}\}$ defined in section \ref{conspos}.
}

\subsection{Completing the proof}

Now, if $c < \frac32 a$, then $ \frac{d}{2} + \frac{d}{a} + \frac{c}{3a}- \frac{1}{2} >   \frac{d}{2} + \frac{d}{a}+\frac{2c}{3a}  - 1     $.

\bigskip

We can combine lemmas \ref{Estdompos} and \ref{Weylpos} as well as {inequality} (\ref{Wlapos}) to get the following:

\begin{lemma}
	If $a,b,c$ are chosen such that 
	\begin{equation}\label{eqCondA}
	\frac{d}{2} + \frac{d}{a} + \frac{c}{3a} - \frac{1}{2} < \frac{d}{2} + \frac{d}{b} \, , 
	\end{equation}
	 then $c<3a/2$ and we have the following estimates:
	
	\begin{align}
		\mu(f_\la) \leq { K_d^{-\frac{d}{2}}} (1+2\delta)^d W_\la \left( 1 + o_\la(1) \right)\,,\\
		N(\la) \geq (2\pi)^{-d}w_d W_\la \left( 1 + o_\la(1) \right)\,.
	\end{align}

\end{lemma}
Now, since $N(\la_n) \leq n$, we get the following estimate:

\begin{align}
	\limsup\limits_{n \to \infty}\frac{\mu(f_n)}{n} \leq \limsup\limits_{n \to \infty}\frac{\mu(f_{\la_n})}{N(\la_n)} \leq \limsup\limits_{n \to \infty} \frac{{ K_d^{-\frac{d}{2}}} (1+2\delta)^d W_\la \left( 1 + o_n(1) \right)}{(2\pi)^{-d}w_d W_\la \left( 1 + o_n(1) \right)} \leq (1+2\delta)^d \gamma_d \, ,
\end{align}

\bigskip

where $\gamma_d$ is Pleijel's constant.

\bigskip

Now, since this estimate is true for any $\delta >0$, we can let $\delta \to 0$ to obtain {Theorem \ref{thma}} (we get the sufficient condition by subtracting $d/2$ from both sides of \eqref{eqCondA}).

\section{Potential vanishing at infinity}

\subsection{Construction of the partition}\label{consneg}

We want to construct a partition of unity of $\mathbb{R}^d$ with the following properties:

\begin{enumerate}
	\item The overlap between the supports of any two functions is small.
	\item Each $x \in \mathbb{R}^d$ is contained in the support of finitely many elements of the partition.
	\item If $x$ is in the support of some function, the diameter of the support will be $\asymp |x|^q$ with $q < 1$.
	\item The gradient of each function in the partition can be bounded by $C |x|^{-q}$, where $x$ is any point in the support of the function.
\end{enumerate}

We will start with a simple lemma that enables us to construct such a partition:

\begin{lemma} \label{lemSeq} For any $q<1$  and $r>0$, there exist a sequence of pairs of positive numbers $(r_i,d_i)$, $i\ge1$ and positive constants $Q,Q'$ such that
	\begin{enumerate}
		\item $r_1=r>0$, 
		\item $r_{i+1}=r_i+d_i$,
		\item $\frac{r_i}{d_i}\in \mathbb N$,
		\item $Qr_i^q\le d_i\le Q'r_i^q$,
		\item $r_i\to\infty$.
	\end{enumerate}
\end{lemma}

The proof of this lemma will be done in appendix \ref{annex3}.

\medskip

We will now use this sequence (with $q$ still undetermined) to construct a covering of $\mathbb{R}^d$. 

\medskip

Let $r=r(V) > 0$ be large enough such that all the singularities of $V$ are contained in the ball of radius $r$ centered at 0.

\medskip

Let $ D(x_1, x_2 \ldots x_d): = \sup |x_d|$. This norm is equivalent to the standard Euclidean norm.

\medskip

Let $ C_i := \left\{ x \, | \,  r_i \leq D(x) \leq r_i + d_i \right\}$ for $i \geq 1$ {and}  $B_0:=\left\{ x \, |  D(x) \leq r \right\}$ .

\bigskip

We now cover $C_i$ with hypercubes $B_{i,j}$ of  {side-length} $d_i$. Since $r_i / d_i$ is an integer, the cubes fit perfectly. We will use an abuse of notation and label the $B_{i,j}$ as simply $B_i$. 

\bigskip

On figure \ref{figPart} is shown an example of such a partition around the origin.

\vspace{0.5cm}

\begin{figure}[H]
\label{figSquares}
	\centering
	\includegraphics[width=8cm, height=8cm]{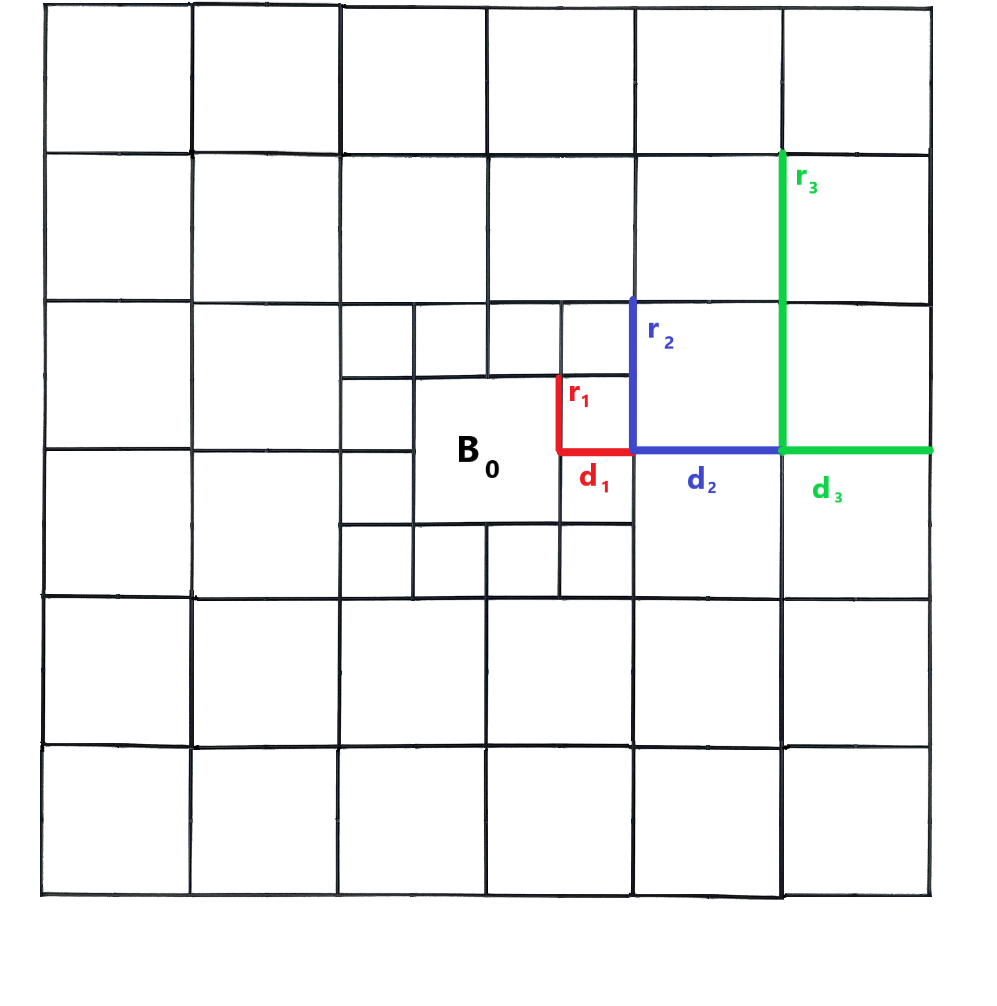}
	\caption{Partition for case B.}
	\label{figPart}
\end{figure}
\bigskip

Let $x_i$ be the center of each hypercube $B_i$. 

\bigskip

Now, let us construct a partition of unity of $\mathbb{R}^d$. Let $ {B} = [-\frac{1}{2}, \frac{1}{2}]^d$ and  $ {B_\delta}=  (-\frac{1}{2} - \delta, \frac{1}{2}+\delta)^d$. There exists a smooth function $\phi_\delta$ with the following properties:

\begin{enumerate}
	\item $\phi_\delta= 1$ in ${ B}$.
	\item  $\mbox{supp}(\phi_\delta)\subset { B_\delta}$.
	\item $\sup |\nabla \phi_\delta| \leq C(d) \delta^{-1}$.
\end{enumerate}

\bigskip

Let $B_i$ be one hypercube, let $\bar{\phi_{i}}$ be defined as the rescaled and translated $\phi_\delta$ such that $\bar{\phi_{i}} = 1$ on $B_i$ and let $B_{i,\delta}$ be defined as the rescaled and translated $B_\delta$ such that $\mbox{supp}(\phi_{i})\subset B_{i,\delta}$. Therefore, the collection $\{B_{i,\delta}\}$ form a covering of $\mathbb{R}^d$.

\bigskip

 For $\delta>0$ small enough, the multiplicity of $\{B_{i,\delta}\}$ is at most $2^d$. This means that for each $x\in\mathbb{R}^d$ there are at most $2^d$ indices $i$ such that $ x \in B_{i,\delta}$. Indeed, since the ratio $d_{i+1}/d_i$ is bounded from above and away from $0$ (see remark \ref{rmrkQ}), choosing $\delta$ small enough insures that any $x$ will only be covered by hypercubes from at most two successive rows.

\bigskip

With the rescaling, {for} $i\neq 0$, $\sup |\nabla \bar{\phi_{i}}| \leq C(d,r,\delta) \sup\limits_{B_{i,\delta}} |x|^{-q}$.

\bigskip

Since the closure of the union of all $B_i$  covers $\mathbb{R}^d$, we have $ \sum \bar{\phi_{i}}(x)^2 \geq 1$ for all $x \in \mathbb{R}^d$. Therefore, the functions $A_{i}:= {\bar{\phi_{i}}}{\left(\sum\limits_i \bar{\phi_{i}}^2\right)^{-\frac{1}{2}}}$, form a partition of unity, and they have the following property:

\bigskip

\begin{align}
	\sup |\nabla A_{i}| \leq C(d,r,\delta) \sup\limits_{B_{i,\delta}}|x|^{-q}\,.
\end{align} 

\bigskip

\subsection{Nodal count}\label{secneg}

{Note that in the rest of this section, $a$, $b$, $c$, $C_1$, $C_2$, $C_3$ denote the constants appearing in definition \ref{caseb}.}
Let  ${M:=}2^d$. We know that for any $\delta >0$ and {any} given nodal domain $\Om$, there exists at least one $i$ such that $\Om$ is {$M$-localized} in $B_{i,\delta}$.

\bigskip

Now, since $\sup\limits_{B_{0,\delta}}[\la-V] = +\infty$, we need to treat the central cube differently than the others. However, our choice of $r$ allows us to control the number of nodal domains that are {$M$-localized} in $B_{0,\delta}$:

\begin{lemma}
If $r$ is large enough that all the singularities of $V$ are contained in the ball of radius $r$ centered at zero, then the number $\mu_0$ of nodal domains which are $M$-localized in $B_{0,\delta}$ is bounded by a constant $C$ which depends on $V, d, r$ and $\delta$, but not on $\la$.
\end{lemma}

\begin{proof} 
Let $\Omega$ be a nodal domain of $f_\la$ which is $M$-localized in $B_{0,\delta}$ and $\Om_0 = \Om \cap B_{0,\delta}$. From identity \eqref{eqIBP}, we have
\begin{equation*}
	\int_{\Om_0}|\nabla(A_0f_\la  )|^2=\int_{\Om_0}(\la-V(x))(A_0f_\la)^2+|\nabla A_0|^2f_\la^2 \, ,
\end{equation*}
so that
\begin{align*}
	\frac12\int_{\Om_0}|\nabla(A_0f_\la  )|^2=\int_{\Om_0}\left((\la-V(x))(A_0f_\la)^2-\frac12|\nabla(A_0f_\la  )|^2+|\nabla A_0|^2f_\la^2\right)\,.
\end{align*}

Recalling that $\la<0$ we get
\begin{align*}
	\int_{\Om_0}|\nabla(A_0 f_\la  )|^2 \le 2\int_{\Om_0}|\nabla A_0|^2f_\la^2-\int_{\Om_0}  |\nabla(A_0f_\la)|^2+2V(x)(A_0f_\la)^2  \,.
\end{align*}

From lemma \ref{lemLSB} and remark \ref{remLSB}, applied to the potential $2V$, there exists $C(d,V)$ (depending only on $d$ and $V$) such that  
\begin{align*}
	\int_{\Om_0}|\nabla(A_0f_\la)|^2+2V(x)(A_0f_\la)^2 \geq -{C(d,V)}\int_{\Om_0} (A_0f_\la)^2
\end{align*} 
and therefore
\begin{align*}
	\int_{\Om_0}|\nabla(A_0f_\la  )|^2\le 2\int_{\Om_0}|\nabla A_0|^2f_\la^2+{C(d,V)} \int_{\Om_0}(A_0 f_\la)^2\,.
\end{align*}

\bigskip

Since $\Omega$ is {$M$-localized} in $B_{0,\delta}$, we get $R_{\Om_0}(A_0f_\la)\le C$, with
\begin{equation*}
	C:={C(d,V)}+2{ M}\sup_{\R^d}|\nabla A_0|^2\,,
\end{equation*}
which depends on $d,  V,  r$ and $\delta$,  but not on $\la$.

\bigskip

 Applying {the Faber-Krahn} inequality and dividing the volume of $B_{0,\delta}$ by the minimal volume of any $\Om_0$, we find
\begin{equation*}
	\mu_0\le (C{ K_d^{-1}})^{\frac{d}2}|B_{0,\delta}|\,.
\end{equation*}

Since $|B_{0,\delta}|$ depends only on $d, r$ and $\delta$, this is the desired result.
\end{proof}

\bigskip
We now turn our attention to the rest of the hypercubes. As in the previous section, using {the Faber-Krahn} inequality and summing over all $i$ we obtain an upper bound for the number of nodal domains of $f_\la$:

\begin{align}
	\mu(f_\la) \leq  \mu_0 + { K_d^{-\frac{d}{2}}}\sum_{i\geq 1}  |B_{i,\delta}|   \sup_{B_{i,\delta}} \left[ \la - V + C(d,\delta)|x|^{-2q} \right]_+^{\frac{d}{2}}\,,\\
	\leq \mu_0 + { K_d^{-\frac{d}{2}}}\sum_{i\geq 1} (1+2\delta)^d |B_{i}|   \sup_{B_{i,\delta}} \left[ \la - V + C(d,\delta)|x|^{-2q} \right]_+^{\frac{d}{2}}\,.
\end{align}

\bigskip
{Again, let 
\[\tilde{M}_\la := \sum_{i\geq 1}  |B_i|  \sup_{B_{i,\delta}} \left[ \la - V +C(d,\delta) |x|^{-2q} \right]_+^{\frac{d}{2}}\]
and  
\[\tilde{W}_\la := \int_{\mathbb{R}^d \backslash B_{0,\delta}}(\la-V)_+^{\frac{d}{2}} \, .\] 

\begin{lemma}\label{lemdomneg1}
	Assuming $c>3a/2$ and for a fixed $\delta >0$, 
	
	\begin{align}
		\mu(f_\la) \leq  \mu_0 + (1+2\delta)^d K_d^{-\frac{d}{2}} \tilde{M}_\la+ C(d,V,\delta) |\la|^{\frac{d}{2}- \frac{d}{a} + \frac{2c}{3a} -1 }\,.
	\end{align}
\end{lemma}

The proof of this lemma will be done in section \ref{annex21}.

\bigskip

Now, we estimate the size of $\tilde{W}_\la$ from below. {For $R>0$ large enough so that $B_{0,\delta}\subset B(0,R)$,}

\begin{align}
	\tilde{W}_\la\geq \int\limits_{ B\left(0, \left(\frac{C_2}{|\la|}\right)^{1/b}\right)\setminus B(0,R)}\left[ \la + C_2 { |x|}^{-b} \right]^{\frac{d}{2}}{\, dx}\\ \label{ineqwneg1}
	\geq {C(V,d)} |\la|^{\frac{d}{2}- \frac{d}{b}}\,.
\end{align}
We also note that since $b < 2$, $\mu_0 = o(\tilde{W}_\la)$.

\subsection{Weyl's law for negative exponents}\label{weylneg}

As before, let  $ \tilde{N}(\lambda) := \sharp \{ \lambda_n < \lambda  \}$.

\begin{lemma}\label{lemweylneg1}
	Assuming $ c > \frac32a$, we obtain the following estimate for $\tilde{N}(\la)$:
	
	\begin{align}
		\tilde{N}(\la) \geq (2\pi)^{-d}w_d \tilde{W}_\la + O( |\la|^{ \frac{d}{2} - \frac{d}{a} + \frac{c}{3a}- \frac{1}{2}     })\,.
	\end{align}

\end{lemma}

The proof of this lemma will be given in section \ref{annex22} using a Dirichlet bracketing argument, based on the partition $\{B_i\}$ defined in section \ref{consneg}.}

\subsection{Completing the proof}

Now, if $ \frac{c}{3a} > \frac{1}{2}$, then ${\frac{d}{2}- \frac{d}{a} + \frac{c}{3a} -\frac{1}{2} }< {\frac{d}{2}- \frac{d}{a} + \frac{2c}{3a} -1 }$. 

\bigskip

We can combine lemmas \ref{lemdomneg1} and \ref{lemweylneg1} as well as inequality (\ref{ineqwneg1}) to get the following:

\begin{lemma}
	If $a,b,c$ are chosen such that 
	\begin{equation}\label{eqCondB}
	\frac{d}{2} - \frac{d}{a} + \frac{c}{3a} - \frac{1}{2} >  \frac{d}{2}-\frac{d}{b}\,,
	\end{equation}
	 then $\frac{c}{3a}>\frac12$ and we have the following estimates:
	
	\begin{align}
	\mu(f_\la) \leq { K_d^{-\frac{d}2}} (1+2\delta)^d \tilde{W}_\la \left( 1 + { o}_\la(1) \right)\,,\\
		\tilde{N}(\la) \geq (2\pi)^{-d}w_d \tilde{W}_\la \left( 1 + { o}_\la(1) \right)\,.
	\end{align}
\end{lemma}

Now, since $\tilde{N}(\la_n) \leq n$,

\begin{align}
	\limsup\limits_{n \to \infty}\frac{\mu(f_n)}{n} \leq \limsup\limits_{n \to \infty}\frac{\mu(f_{\la_n})}{\tilde{N}(\la_n)} \leq \limsup\limits_{n \to \infty} \frac{{ K_d^{-\frac{d}2}} (1+2\delta)^d \tilde{W}_{\la_n} \left( 1 + { o}_n(1) \right)}{(2\pi)^{-d}w_d \tilde{W}_{\la_n} \left( 1 + { o}_n(1) \right)} \leq (1+2\delta)^d \gamma_d \, ,
\end{align}

\bigskip

where $\gamma_d$ is Pleijel's constant.

\bigskip

Now, since this estimate is true for any $\delta >0$, we can let $\delta \to 0$ to obtain theorem \ref{thmb} (we get the sufficient condition by subtracting $d/2$ from both sides of \eqref{eqCondB}).

\appendix

\section{ Proof of Lemma \ref{lemLSB} and Proposition \ref{propSpectrum} }
\label{annex0}

Let us first recall the following well-known Hardy-type inequalities (see for instance \cite[Sec. 1.2, Examples 1 and 2]{Cow2010}).

\bigskip

\begin{lemma}\label{lemHardy} First, we assume $d\ge 3$. Then, for any $\varphi\in C_c^\infty(\R^d)$,
\begin{equation}\label{eqHardyd}
	\int |\nabla \varphi|^2\ge \int h_d(x)|\varphi|^2\,,
\end{equation}
where 
\begin{equation*}
	h_d(x):=\left(\frac{d-2}2\right)^2\frac1{|x|^2}\,.
\end{equation*}
Then, we assume $d=2$. Then, for any $R>0$ and $\varphi\in C_c^\infty(D(0,R))$,
\begin{equation}\label{eqHardy2}
	\int|\nabla \varphi|^2\ge \int h_{2,R}(x)|\varphi|^2\,,
\end{equation}
where $D(0,R)\subset\mathbb{R}^2$ is the open disk of radius $R$ centered at $0$ and
\begin{equation*}
	h_{2,R}(x):=\frac1{ 4|x|^2\ln^2\left(\frac{|x|}{R}\right)}\,.
\end{equation*}

\end{lemma}

\bigskip

We can now proceed with the proof of lemma \ref{lemLSB}. In case A, $V\ge0$ and the result is obvious (with $C=0$). Let us consider case B. Recall that we denote by $X=\{x_1,\dots,x_N\}$ the set of singular points of $V$. Let us treat separately the cases $d\ge3$ and $d=2$. 

\bigskip

In  case $d\ge3$, we have, by translating the inequality \eqref{eqHardyd}, 
\[\int |\nabla u|^2 \ge \int h_d(x-x_i)|u|^2\]
for all $1\le i \le N$, hence
\begin{align}\label{eqHardyMult}
	\int |\nabla u|^2 \ge \int h(x)|u|^2\,,
\end{align}
with
\begin{align}\label{eqHardyWeight}
	h(x):=\frac{1}{N} \sum_{i=1}^N h_d(x-x_i)\,.
\end{align}
From the behavior of $V$ at each $x_i$, given in definition \ref{caseb}, $\lim_{x\to x_i}h(x)+V(x)=+\infty$. It follows that $h+V$ is bounded from below. We can then take $C=-\inf_{\mathbb{R}^d}(h+V)$. 
\bigskip

In case $d=2$, we choose for each $1\le i\le N$ a radius $r_i>0$, such that the open disks $D(x_i,2r_i)$ are pairewise disjoint. In addition, we choose smooth functions  $\chi_i$, compactly supported in $D(x_i,2r_i)$, such that $\chi_i\le 1$ in $\mathbb{R}^2$ and $\chi_i=1$ in $ D(x_i,r_i)$. We have pointwise
\[|\nabla (\chi_i u)|^2\le 2\chi_i^2|\nabla u|^2+2|\nabla \chi_i|^2|u|^2\le 2|\nabla u|^2+2|\nabla \chi_i|^2|u|^2\,,\]
and therefore
\[\int |\nabla u|^2\ge \frac{1}{2}\int |\nabla (\chi_i u)|^2-\int |\nabla \chi_i|^2|u|^2\,.\]
Using inequality \eqref{eqHardy2},
\[\int|\nabla u|^2\ge\int\left(\frac12\chi_i^2h_{2,2r_i}(x-x_i)-|\nabla \chi_i|^2\right)|u|^2\,.\]
Summing over all $i$'s, we find
\begin{align}\label{eqHardyMult2D}
\int |\nabla u|^2\ge \int h(x)|u|^2\,,
\end{align}
where
\begin{align}\label{eqHardyWeight2D}
h(x):=\frac1N\sum_{i=1}^N\left(\frac12\chi_i^2h_{2,2r_i}(x-x_i)-|\nabla \chi_i|^2\right).
\end{align}
As in the case $d\ge3$, $\lim_{x\to x_i}h(x)+V(x)=+\infty$ for all $1\le i\le N$, and therefore $h+V$ is bounded from below and we can take $C=-\inf_{\mathbb R^2}(h+V)$.

\bigskip

The analysis leading to proposition \ref{propSpectrum} is rather classical (see for instance \cite{Agm}). For the sake of completeness and the convenience of the reader, let us recall the main steps of the proof. In cases A and B, the potential $V$ belongs to $L^p_{loc}(\R^d)$ for some $p>d/2$. Indeed, it suffices to choose $p$ such that $p>d/2$ and ${a_i}\,p<d$ for all $1\le i \le N$. According to \cite[Th. 3.2, pp. 44-45, Eq. (3.16)]{Agm}, we can then apply Persson's formula to compute the bottom of the essential spectrum of $H_V$:
\begin{equation}
\label{eqPersson}
	\inf \sigma_{{ess}}({H_V})=
	\sup_{K}\inf_{
	\begin{array}{c}
	\varphi \in C_c^\infty(\R^d\setminus K)\\
	 \varphi\neq0
	 \end{array}}
	 \frac{\langle \varphi, { H_V}\varphi\rangle}{\|\varphi\|^2}\,,
\end{equation}
where $K$ ranges over all compact subsets of $\R^d$. It follows immediately that 
\begin{equation*}
	\inf \sigma_{{ess}}({H_V})\ge \sup_{K} \inf_{\R^d\setminus K}V\,.
\end{equation*}

In case A, $V(x)\to+\infty$ as $|x|\to +\infty$, and thus
\begin{equation*}
	\inf \sigma_{{ess}}({H_V})=+\infty\,,
\end{equation*}
that is to say $\sigma_{{ess}}({ H_V})=\emptyset$. Since ${H_V}$ is lower semi-bounded, $\sigma({H_V})$ is a sequence of eigenvalues with finite multiplicities tending to $+\infty$.

\bigskip

In case B, $V(x)\to 0$ as $|x|\to +\infty$ and we obtain 
\begin{equation*}
	\sigma_{{ess}}({H_V})\subset[0,+\infty)\,.
\end{equation*} 

\bigskip

We can easily check the reverse inclusion by constructing appropriate Weyl sequences. Indeed, let us set, for $\xi\in \R^d$,
\begin{equation*}
	\varphi_{\xi,n}(x):=C_n\chi\left(\frac{x-x_n}{R_n}\right)\exp(i \xi\cdot x)\,,
\end{equation*}
where $\chi$ is a smooth function in $\R^d$ with compact support, $x_n\in\R^d$, $R_n>0$, and $C_n>0$ is defined by $\|\varphi_{\xi,n}\|=1$. By a suitable choice of sequences $|x_n|\to+\infty$ and $R_n\to +\infty$, we can ensure that $\langle \varphi_{\xi,n},{ H_V}\varphi_{\xi,n}\rangle\to|\xi|^2$ and $(\varphi_{\xi,n})$ converges to $0$ weakly in $L^2(\R^d)$, which implies $|\xi|^2 \in \sigma_{ess}({ H_V})$. We have shown that $\sigma_{{ess}}({H_V})=[0,+\infty)$.

\bigskip

For the characterization of the form domain, we refer the reader to \cite[Th. 8.2.1]{Dav1995} for case A and \cite[Th. 8.2.3]{Dav1995} for case B. Note that in the latter case, the result is given for $d\ge3$. The proof  can be extended to $d=2$ using inequality \eqref{eqHardyMult2D}.

\section{Proofs of lemmas \ref{Estdompos} and \ref{Weylpos} }\label{annex1}

\subsection{Proof of lemma \ref{Estdompos}} \label{annex11}
Recall that the {side-length} of the cubes in the partition is $\rho (1+2\delta)$ with $\rho = \la^{-m}$.
Also, we {have} the following conditions on $V$:

\begin{align}
	V(x) \geq C_1 |x|^a \, ,\\
	V(x) \leq C_2 |x|^b \, ,\\
	|\nabla V(x)| \leq C_3 |x|^c \, ,
\end{align}

with $c < \frac{3}{2}a$.

\bigskip

{Recall also that 
\[M_\la := \sum_{z \in \mathbb{Z}^d} \la^{-md}  \sup_{J_{z,\delta,\rho}} \left[ \la - V + C\la^{2m} \right]_+^{\frac{d}{2}}\]
and
\[W_\la := \int_{\mathbb{R}^d}(\la-V)_+^{\frac{d}{2}}\,.\]
We also define
\[m_\la := \sum_{z \in \mathbb{Z}^d} \la^{-md} \inf_{J_{z,\delta,\rho}} \left[ \la - V  \right]_+^{\frac{d}{2}}\] and 
\[A_\la := M_\la - W_\la\,.\]

\bigskip

Since $A_\la\leq M_\la-m_\la$, it suffices to show that  $A_\la	\leq C(d,V,\delta) { \la^{ \frac{d}{2} + \frac{d}{a} + \frac{ 2c}{3a} - 1     }}$  in order to prove lemma \ref{Estdompos}.}

\bigskip

 In this section, $C$ will denote constants which may change from line to line but which never depend on $\la$. However we may write down the dependency of $C$ on different parameters to emphasize this point. {The same remark applies to the ``O" and ``o" occuring in the formulas.}

\bigskip

We have the following estimates for $A_\la$:

\begin{align}
\nonumber	A_\la \leq M_\la - m_\la \leq \la^{-md} \sum_{z \in \mathbb{Z}^d} \left[ \sup_{J_{z,\delta,\rho}} \left[ \la - V +{\rd C}\la^{2m} \right]_+^{\frac{d}{2}} - \inf_{J_{z,\delta,\rho}} \left[ \la - V \right]_+^{\frac{d}{2}} \right]\\
	\leq \la^{-md}  \sum_{z \in \mathbb{Z}^d} \left[ \sup_{J_{z,\delta,\rho}} \left[ \la - V  \right]_+^{\frac{d}{2}}- \inf_{J_{z,\delta,\rho}} \left[ \la - V \right]_+^{\frac{d}{2}}   + C(d,\delta)\la^{2m}\sup_{J_{z,\delta,\rho}}\left[ \la - V + C\la^{2m}\right]_+^{\frac{d}{2}-1} \right] \, ,
\end{align}

\noindent where we have used the {inequality} $(a+b)^k \leq a^k + k b (a+b)^{k-1}$ for any $a,b,k \in \mathbb{R}^+$, with $k\geq1$. We now bound the oscillation of $[\la - V]^{\frac{d}{2}}$ over a cube using farthest distance between two points and the bounds on the gradient, recalling that the diameter of the cubes is less or equal than $C(d) \la^{-m}$:

\begin{align}
\nonumber \la^{-md}  \sum_{z \in \mathbb{Z}^d} \left[ \sup_{J_{z,\delta,\rho}} \left[ \la - V  \right]_+^{\frac{d}{2}}- \inf_{J_{z,\delta,\rho}} \left[ \la - V \right]_+^{\frac{d}{2}}   + C(d,\delta)\la^{2m}\sup_{J_{z,\delta,\rho}}\left[ \la - V + C\la^{2m}\right]_+^{\frac{d}{2}-1} \right]	\\ \leq \la^{-md}  \sum_{z \in \mathbb{Z}^d} \left[ C(d) \la^{-m} \sup_{J_{z,\delta,\rho}}|\nabla V| \sup_{J_{z,\delta,\rho}}\left[ \la - V  \right]_+^{\frac{d}{2}-1} + C(d,\delta)\la^{2m}\sup_{J_{z,\delta,\rho}}\left[ \la - V + C\la^{2m}\right]_+^{\frac{d}{2}-1} \right] \, . 
\end{align}
To estimate the last term we use the {inequality} $(a+b)^k \leq 2^k a^k + 2^k b^k$:

\begin{align}
\nonumber \la^{-md}  \sum_{z \in \mathbb{Z}^d} \left[ C(d) \la^{-m} \sup_{J_{z,\delta,\rho}}|\nabla V| \sup_{J_{z,\delta,\rho}}\left[ \la - V  \right]_+^{\frac{d}{2}-1} + C(d,\delta)\la^{2m}\sup_{J_{z,\delta,\rho}}\left[ \la - V + C\la^{2m}\right]_+^{\frac{d}{2}-1} \right]\\  \label{3termpositive}
	\leq \la^{-md}\sum_{z \in \mathbb{Z}^d} \left[ C(d) \la^{-m} \sup_{J_{z,\delta,\rho}}|\nabla V| \sup_{J_{z,\delta,\rho}}\left[ \la - V  \right]_+^{\frac{d}{2}-1}  +
	  C(d,\delta) \la^{2m} \sup_{J_{z,\delta,\rho}} \left[ \la - V \right]_+^{\frac{d}{2}-1} \right] \\ \nonumber + \sum_{\substack{z \in \mathbb{Z}^d\\ \sup\limits_{J_{z,\delta,\rho}} V \leq \la +  C\la^{2m}    }} C(d,\delta)\,.
\end{align}

We now assume that $m< \frac12$ so that $\la^{2m} = o(\la)$. As we will see later, 
we can find such an $m$ for suitable $a$, $b$ and $c$.

\bigskip

Since $V(x) \geq C_1 |x|^a$, the second sum in equation (\ref{3termpositive}) can be estimated by the number of cubes of {side-length} $\la^{-m}$ in a ball of radius $C(d,V) \la^{\frac1a}$. We can estimate it directly:

\begin{align}\label{eq41}
	\sum_{\substack{z \in \mathbb{Z}^d\\ \sup\limits_{J_{z,\delta,\rho}} V \leq \la + C\la^{2m}    }}  C(d,\delta) \leq C(d,V,\delta) \la^{\frac{d}{a}+ dm}\,.
\end{align}

Now, by the definition of $J_{z,\delta, \rho}$, we get that for all exponents $M >0$, there exists a constant $C(d,M)$ such that 

\begin{align}
	|J_{z,\delta,\rho}|	\sup\limits_{ x \in { J_{z,\delta,\rho}}} |x|^M \leq C(d,M) \int_{{ J_{z,\delta,\rho}}} |x|^M{\,dx} \,.
\end{align}

We can now bound the first sum in inequality \eqref{3termpositive} using the bounds on $V$ and $\nabla V$:

\begin{align}
\nonumber	\la^{-md}\sum_{z \in \mathbb{Z}^d} \left[ C(d) \la^{-m} \sup_{J_{z,\delta,\rho}}|\nabla V| \sup_{J_{z,\delta,\rho}}\left[ \la - V  \right]_+^{\frac{d}{2}-1}  +
	C(d,\delta) \la^{2m} \sup_{J_{z,\delta,\rho}} \left[ \la - V \right]_+^{\frac{d}{2}-1} \right]	\\ \leq C(d,V,\delta)  \int\limits_{B(0, (\la/C_1)^{1/a})} \left[ \la^{-m} {|x|}^c + \la^{2m}  \right] \left[ \la - C_1 {|x|}^a  \right]^{\frac{d}{2}-1} \,dx \\
	\leq C(d,V,\delta) \left[ \la^{-m + \frac{c}{a} + \frac{d}{a} + \frac{d}{2} - 1} + \la^{2m + \frac{d}{a}+ \frac{d}{2}-1} \right]\,. \label{eq59}
\end{align}

In order to balance the two terms on the right-hand side of inequality (\ref{eq59}), we need that $$-m + \frac{c}{a} = 2m\,.$$

\bigskip

Therefore, putting $m = \frac{c}{3a}$ and combining equations (\ref{eq41}) and (\ref{eq59}) gives us the following estimate for $A_\la$:

\bigskip

\begin{align}
	A_\la	\leq C(d,V,\delta) \left[ \la^{\frac{2c}{3a} + \frac{d}{a} + \frac{d}{2} - 1     } + \la^{\frac{d}{a} + \frac{dc}{3a}}   \right]\,.
\end{align}	

\bigskip

Now, in the case $d=2$, {the exponents in the two terms are equal}. In the case $d\geq 3$, since we assumed that $m < \frac{1}{2}$, the term on the left dominates and we obtain that for any dimension,  

\begin{align}
	A_\la \leq C(d,V,\delta) \la^{\frac{2c}{3a} + \frac{d}{a} + \frac{d}{2} - 1     } \, . \label{boundA}
\end{align}

\bigskip

This completes the proof of lemma \ref{Estdompos}.

\bigskip

\subsection{Proof of lemma \ref{Weylpos}} \label{}

Let $R_\Om$ be the Rayleigh quotient for the Dirichlet Laplacian on $\Om$ and $Q_\Om (f)$ the modified Rayleigh quotient for $H = -\Delta +V$:

$$Q_\Om(f) = \frac{<-\Delta f,f>_{L^2(\Om)}+<V f,f>_{L^2(\Om)}}{<f,f>_{L^2(\Om)}}.$$

\bigskip

{Recall that $J_{z,\rho}$ denotes the hypercubes} $\rho(z_1-\frac{1}{2} , z_1+\frac{1}{2}) \times \rho(z_2 - \frac{1}{2} , z_2+\frac{1}{2}) \ldots \rho(z_d - \frac{1}{2} , z_d + \frac{1}{2} )$. Since they do not overlap, we can use them to do a Dirichlet bracketing for the eigenvalue count.  More specifically, we use the following facts. First,
	\begin{align*}N(\la)=\max\{\dim(V)\,|\,V\subset H^1(\mathbb{R}^d) \mbox{ such that }Q_{\mathbb{R}^d}(f)<\la \mbox{ for all } f\in V\setminus\{0\}\}\,,
	\end{align*}
	with $V$ a vector space. Second,
	\begin{align*}
		N(\la)\ge\sum_{z\in\mathbb{Z}^d}\hat{N}_z(\la)\,,
	\end{align*}
	where
	\begin{equation*}
		\hat{N}_z(\la)=\max\left\{\dim(V)\,|\,V\subset H^1_0(J_{z,\rho}) \mbox{ such that } Q_{J_{z,\rho}}(f)<\la \mbox{ for all } f\in V\setminus\{0\}\right\}\,.
	\end{equation*}
	Finally, we obviously have, for all $z\in\mathbb{Z}^d$,
	\begin{equation*}
		\hat{N}_z(\la)\ge N_z(\la)\,,
	\end{equation*}
	where $		N_z(\la)=\max \left\{\dim(V)\,|\,V\subset H^1_0(J_{z,\rho}) \mbox{ such that }R_{J_{z,\rho}}(f)<\la-\sup_{J_{z,\rho}} V \mbox{ for all } f\in V\setminus\{0\} \right\}$.
	
	\bigskip
	
	We recognize $N_z(\la)$ as  the eigenvalue count below $\la-\sup_{J_{z,\rho}}V$ for the Dirichlet Laplacian on the hypercube $J_{z,\rho}$. We can use the explicit formula found in \cite[lemma 2.4]{GitLar2017} to get :
	
	\begin{align}
		N_z(\la)= |J_{z,\rho}| (2\pi)^{-d}w_d \inf\limits_{J_{z,\rho}}(\la-V)_+^{\frac{d}{2}} + O\left(|\partial J_{z,\rho}| \inf\limits_{J_{z,\rho}}(\la-V)_+^{\frac{d}{2} - \frac{1}{2}}\right)  \, ,\label{eq26}
	\end{align}

\noindent where $w_d$ is the volume of the unit ball in $\mathbb{R}^d$. 

\bigskip

Recall that we want to show that $N(\la) \geq (2\pi)^{-d}w_d W_\la + O( \la^{\frac{d}{2} + \frac{d}{a} + \frac{c}{3a} - \frac{1}{2}     })$.  By summing equation (\ref{eq26}) over all cubes, we obtain the following estimate:

\begin{align}
	N(\la) \geq (2\pi)^{-d}w_d\sum_{z \in \mathbb{Z}^d}  |J_{z,\rho}|  \inf\limits_{J_{z,\rho}}(\la-V)_+^{\frac{d}{2}} + O\left(\sum_{z \in \mathbb{Z}^d}|\partial J_{z,\rho}| \inf\limits_{J_{z,\rho}}(\la-V)_+^{\frac{d}{2} - \frac{1}{2}}\right)\\
	\geq (2 \pi)^{-d}w_d m_\la + O\left(\sum_{z \in \mathbb{Z}^d}|\partial J_{z,\rho}| \inf\limits_{J_{z,\rho}}(\la-V)_+^{\frac{d}{2} - \frac{1}{2}}\right) \label{eq58} \, ,
\end{align}

\bigskip
\noindent where $m_\la$ is the same as in section \ref{annex11}. Since $m_\la = W(\la) + O(A_\la)$, by choosing $m=c/3a$ the first term in (\ref{eq58}) is equal to $(2\pi)^{-d}w_d W_\la + O(\la^{\frac{2c}{3a} + \frac{d}{a} + \frac{d}{2} - 1     })$ by inequality  (\ref{boundA}). 

\bigskip
We now bound the second term:

\begin{align}
\nonumber	\sum_{z \in \mathbb{Z}^d}|\partial J_{z,\rho}| \inf\limits_{J_{z,\rho}}(\la-V)_+^{\frac{d}{2} - \frac{1}{2}}\leq{ C(d)\la^{m}}\sum_{z \in \mathbb{Z}^d}\la^{-md} \inf\limits_{J_{z,\rho}}(\la-V)_+^{\frac{d}{2} -\frac{1}{2}} \\
	\leq { C(d)}\la^m \int\limits_{B(0, (\la/C_1)^{1/a})} (\la-C_1 |x|^a)_+^{\frac{d}{2} -\frac{1}{2}}\,dx \\
	\leq C(d,V) \la^{\frac{c}{3a} + \frac{d}{a} + \frac{d}{2} - \frac{1}{2}     }\,.
\end{align}

\bigskip

Now, since $2c <3a$, then $\frac{c}{3a} - \frac{1}{2}    > \frac{2c}{3a} -1$ and we have the final estimate for $N(\la)$:

\bigskip

\begin{align}
	N(\la) \ge (2\pi)^{-d}w_d W_\la + O \left(\la^{\frac{c}{3a} + \frac{d}{a} + \frac{d}{2} - \frac{1}{2}     }\right)\,.
\end{align}

\bigskip

This completes the proof of lemma \ref{Weylpos}.

\section{Proofs of lemmas \ref{lemdomneg1} and \ref{lemweylneg1}}\label{annex2}

\subsection{Proof of lemma \ref{lemdomneg1}}\label{annex21}

{Recall that the diameter of the hypercubes $B_{i,\delta}$ in the partition is less than $C(d) \sup\limits_{B_{i,\delta}}|x|^q$ and their volume is less than $C(d) \sup\limits_{B_{i,\delta}}|x|^{qd}$. 
	
	\bigskip
	
	Again, let 
\[\tilde{M}_\la := \sum_{i\geq 1}  |B_i|  \sup_{B_{i,\delta}} \left[ \la - V +  C(d,\delta) |x|^{-2q} \right]_+^{\frac{d}{2}}\]
and
\[\tilde{W}_\la := \int_{\mathbb{R}^d \backslash B_{0,\delta}}(\la-V)_+^{\frac{d}{2}}\,.\]
We also define
\[\tilde{m}_\la := \sum_{i\geq 1}  |B_i| \inf_{B_{i,\delta}} \left[ \la - V  \right]_+^{\frac{d}{2}}\]
and
\[\tilde{A}_\la := \tilde{M}_\la - \tilde{W}_\la\,.\]

Finally, recall that 

\begin{align}
	V(x) \geq -C_1 |x|^{-a} \, ,\\
	V(x) \leq -C_2 |x|^{-b} \, ,\\
	|\nabla V(x)| \leq C_3 |x|^{-c} \, ,
\end{align}

\noindent with $c > \frac{3}{2} a$.

\bigskip

Since $\tilde{A}_\la \leq \tilde{M}_\la-\tilde{m}_\la$, it suffices to show that  $\tilde{A}_\la \leq C(d,V,\delta) |\la|^{\frac{d}{2}- \frac{d}{a} + \frac{2c}{3a} -1 }$ in order to prove lemma \ref{lemdomneg1}. }

\bigskip

We now fix $r >0$ such that all the singularities of $V$ are contained in the ball of radius $r$ centered at the origin (if there are no singularities, fix $r=1$).
\bigskip

In this section, $C$ will denote constants which may change from line to line but which never depend on $\la$. However we may write down the dependency of $C$ on different parameters to emphasize this point. {The same remark applies to the ``O" and ``o" occuring in the formulas.}

\bigskip

We will follow a strategy similar to the one we used for positive potentials in order to bound $\tilde{A}_\la$. Again, recalling that the diameter of $B_{i}$ is less than $C(d,\delta)\sup\limits_{B_{i,\delta}}|x|^{q}$, the volume of $B_{i,\delta}$ is less than $C(d,\delta) \sup\limits_{B_{i,\delta}}|x|^{qd} $ and using the {inequalities} $(a+b)^k \leq a^k + kb(a+b)^{k-1}$ and $(a+b)^k \leq 2^ka^k+2^kb^k$, we obtain the following:

\begin{align}
	\tilde{A}_\la  \leq \sum_{i\geq 1} |B_i| \left[  \sup_{B_{i,\delta}} \left[ \la - V + C(d,\delta) |x|^{-2q} \right]_+^{\frac{d}{2}} - \inf_{B_{i,\delta}} \left[ \la - V \right]_+^{\frac{d}{2}} \right]\\
	\leq \sum_{i\geq 1} |B_i| \left[  \sup_{B_{i,\delta}} \left[ \la - V  \right]_+^{\frac{d}{2}}- \inf_{B_{i,\delta}} \left[ \la - V \right]_+^{\frac{d}{2}}   + C(d,\delta) \sup\limits_{B_{i,\delta}}|x|^{-2q}\sup_{B_{i,\delta}}\left[ \la - V + C(d,\delta) |x|^{-2q}\right]_+^{\frac{d}{2}-1} \right]\\
	\leq  \sum_{i\geq 1} |B_i| \left[ C(d,\delta) \sup_{B_{i,\delta}}|x|^{q}\sup_{B_{i,\delta}}|\nabla V| \sup_{B_{i,\delta}}\left[ \la - V  \right]_+^{\frac{d}{2}-1} + C(d,\delta) \sup\limits_{B_{i,\delta}}|x|^{-2q}\sup_{B_{i,\delta}}\left[ \la - V \right]_+^{\frac{d}{2}-1} \right]\label{eq68} \\
+\nonumber \sum_{\substack{{i\geq 1}  \\ { \sup\limits_{B_{i,\delta}} V(x) -C(d,\delta)|x|^{-2q} \leq \la}}} C(d,\delta) 
\end{align}

\bigskip

Now, the last term in (\ref{eq68}) is bounded {(up to a constant factor)} by the number of hypercubes $B_i$ which intersect the region   $ \left\{  V(x) -C(d,\delta)|x|^{-2q} \leq\la\right\}$. In order to estimate this, let $Q(x)$ be the inverse of the volume of the hypercube that contains $x$. From the construction of our partition, $Q(x) \leq C |x|^{-qd}$. {We assume that  $2q>a$,  which will be justified later by a suitable choice of $q$.} Hence, we can bound the last term by the following (recalling that $x_i$ is the center of $B_i$):
\begin{align}
	\sum_{\substack{{i\geq 1} \\  { \sup\limits_{B_{i,\delta}} V(x) -C(d,\delta)|x|^{-2q} \leq \la}}} C(d,\delta) &\leq  C(d,\delta)\sum_{\substack{{i\geq 1}  \\ { \sup\limits_{B_{i,\delta}} -|x|^{-a} \leq C(d,\delta,V) \la}} } Q(x_i)|B_i|\\
	 	&\leq C(d,\delta, V)\int\limits_{B\left(0,C |\la|^{-1/a}\right) \backslash B_0} |x|^{-qd}{\,dx} \\ & \leq C(d,\delta, V) |\la|^{\frac{qd}{a} - \frac{d}{a}}\,. \label{eq2}
\end{align}

We can combine estimates (\ref{eq68}) and \eqref{eq2} as well as bounds on $|V|$ and $|\nabla V|$:

\bigskip

\begin{align}
	\tilde{A}_\la \leq  C(d,\delta, V) |\la|^{\frac{qd}{a} - \frac{d}{a}} + \sum_{i \geq 1} |B_i| C(d,\delta) \sup_{B_{i,\delta}}\left[ \la - V\right]_+^{\frac{d}{2}-1} \left[ \sup_{B_{i,\delta}} |x|^{-2q} + \sup_{B_{i,\delta}} |x|^{q} \sup_{B_{i,\delta}}|\nabla V| \right] \\
	\leq  C(d,\delta, V) |\la|^{\frac{qd}{a} - \frac{d}{a}} + \sum_{i \geq 1} |B_i|  C(d,\delta, V) \sup_{B_{i,\delta}} |x|^{\frac{-ad}{2}+a} \left[ \sup_{B_{i,\delta}} |x|^{-2q}+ \sup_{B_{i,\delta}} |x|^{q} \sup_{B_{i,\delta}} |x|^{-c} \right]\,. \label{eq73}
\end{align}

\bigskip

Now, by the construction of the cubes $B_i$, we have that for every $i\geq1,\delta$ and exponent $P<0$, 
\[\sup_{B_{i,\delta}} |x|^P \leq C(P) \inf_{B_{i,\delta}} |x|^P\,.\]
 Hence, we can estimate the sum on the right-hand side of \ref{eq73} by an integral:

\begin{align}
	\tilde{A}_\la	\leq C(d,\delta, V) |\la|^{\frac{qd}{a} - \frac{d}{a}} +  \int\limits_{B_0(C |\la|^{-1/a}) \backslash B_0(r)} C(d,\delta, V) |x|^{\frac{-ad}{2} +a} \left[|x|^{-2q} + |x|^q |x|^{-c}   \right] {\,dx}\\
	\leq C(d,\delta, V) |\la|^{\frac{qd}{a} - \frac{d}{a}} + C(d,\delta, V) |\la|^{\frac{d}{2}-1 + \frac{2q}{a} - \frac{d}{a}} + C(d,\delta, V) |\la|^{\frac{d}{2}-1 + \frac{c}{a}-\frac{q}{a}- \frac{d}{a}}\,.
\end{align}

\bigskip

Now, taking $q=\frac{c}{3}$ balances the last two terms and we end up with the following:

\begin{align}
		\tilde{A}_\la \leq C |\la|^{\frac{cd}{3a} - \frac{d}{a}} + C |\la|^{\frac{d}{2}- \frac{d}{a} + \frac{2c}{3a} -1 }\,.
\end{align}

\bigskip

Let us recall that $c<3$ (see remark \ref{remc}) and therefore $q<1$. In turn, this makes sure that the conditions of lemma \ref{lemSeq} are fulfilled.  Now,  if $d=2$, both terms are equal. If $d \geq 3$, then the right term dominates and in both cases

\begin{align}
	\tilde{A}_\la \leq C |\la|^{\frac{d}{2}- \frac{d}{a} + \frac{2c}{3a} -1 }\, . \label{eq76} 
\end{align}
\bigskip

Finally, we note that by imposing $c > \frac32 a$,  {we fulfill the condition $2q>a$}. This completes the proof of lemma \ref{lemdomneg1}.

\bigskip

\subsection{Proof of lemma \ref{lemweylneg1}}\label{annex22}

We want to show that $	\tilde{N}(\la) \geq (2\pi)^{-d}w_d \tilde{W}_\la + O( \la^{ \frac{d}{2} - \frac{d}{a} + \frac{c}{3a}- \frac{1}{2}     })$.

\bigskip

As in the case with positive potentials, we will use a Dirichlet bracketing argument for the number of eigenvalues. Let $R_\Om$ be the Rayleigh quotient for the Dirichlet laplacian on $\Om$. Since the hypercubes $B_i$ do not overlap, we have the following lower bound:

	\begin{align} \label{eqweylneg}
		\tilde{N}(\la) \geq \sum_{i\geq 1} \tilde{N}_{i}(\la)\,,
	\end{align}
	where, similarly to section \ref{weylpos},
	\begin{align}
		\tilde{N}_i(\la):=\sup\left\{\dim(V) \,|\, V \subset H_0^1(B_i) \mbox{ such that } R_{B_i}(f) \leq  \inf\limits_{B_i}(\la  - V)  \mbox{ for all } f \in V, \, f \neq 0\right\}\,,
	\end{align}
$V$ being a vector space.

Using inequality \eqref{eqweylneg} and the formula for the Dirichlet spectrum of a hypercube with remainder, we obtain the following:

\begin{align}
		\tilde{N}(\la)		\geq \sum\limits_{i \geq 1} |B_i|(2\pi)^{-d}w_d \inf_{B_{i}}(\la-V)_+^{\frac{d}{2}} + O\left( \sum\limits_{i \geq 1} |\partial B_i| \inf\limits_{B_i} (\la-V)_+^{\frac{d}{2}-\frac{1}{2}}       \right) \, . \label{eq84}
\end{align}

We now compare the first term with $\tilde{W}_\la$:

\begin{align}
	\sum\limits_{i \geq 1} |B_i| \inf_{B_{i}}(\la-V)_+^{\frac{d}{2}} = \tilde{W_\la} + O \left( \sum_{i\geq 1} |B_i|\sup_{B_{i}}(\la-V)_+^{\frac{d}{2}} - \sum_{i\geq 1} |B_i|\inf_{B_{i}}(\la-V)_+^{\frac{d}{2}}     \right)  \\
	= \tilde{W_\la} + O\left(  \sum_{i \geq 1} |B_i| \sup_{B_{i}}|\nabla V| \sup_{B_{i}} (\la - V)_+^{\frac{d}{2}-1}   \right)\,.
\end{align}
Under the assumption $c>\frac32a$, choosing $q=\frac{c}{3}$ in the partition enables us to use the exact same steps as in section \ref{annex21} (the details are left to the reader):

\begin{align}
	\sum\limits_{i \geq 1} |B_i| \inf_{B_{i}}(\la-V)_+^{\frac{d}{2}} = \tilde{W}_\la + O(|\la|^{\frac{d}{2}- \frac{d}{a} + \frac{2c}{3a} -1 })\,.
\end{align}

We now bound the second term in (\ref{eq84}) by noting that due to the construction of the partition, for any exponents $q>0$ (for the size of the cubes) and $P<0$, $\sup_{B_{i}} |x|^P \leq C(q,P) \inf_{B_{i}} |x|^P$:

\begin{align}
	\sum\limits_{i \geq 1} |\partial B_i| \inf\limits_{B_i} (\la-V)_+^{\frac{d}{2}-\frac{1}{2}} &\leq C(d,V)\sum\limits_{i \geq 1} |B_i| \sup\limits_{B_i} |x|^{-\frac{c}{3}} \inf\limits_{B_i} (\la-V)_+^{\frac{d}{2}-\frac{1}{2}}\\
	 &\leq C(d,V)\sum\limits_{i \geq 1}|B_i| \inf\limits_{B_i} |x|^{-\frac{c}{3}} \inf\limits_{B_i} (C|x|^{-a})_+^{\frac{d}{2}-\frac{1}{2}}\\
 &\leq C(d,V)\int\limits_{B(0,C |\la|^{-1/a}) \backslash B_0} |x|^{-\frac{c}{3}}(C|x|^{-a})^{\frac{d}{2}-\frac{1}{2}}\\
	&\leq C(d,V) |\la|^{\frac{d}{2}- \frac{d}{a} + \frac{c}{3a} -\frac{1}{2} }\,.
\end{align}

Since $\frac{c}{3a} > \frac{1}{2}$, then ${\frac{d}{2}- \frac{d}{a} + \frac{c}{3a} -\frac{1}{2} }< {\frac{d}{2}- \frac{d}{a} + \frac{2c}{3a} -1 }$ and $$\tilde{N}(\la) \geq (2\pi)^{-d}w_d \tilde{W_\la} + O({|\la|}^{\frac{d}{2}- \frac{d}{a} + \frac{c}{3a} -\frac{1}{2} })\,.$$

\bigskip

This completes the proof of lemma \ref{lemweylneg1}.

\section{Proof of lemma \ref{lemSeq}}\label{annex3}

We recall lemma \ref{lemSeq}:

 For any $q<1$  and $r>0$, there exist a sequence of pairs of positive numbers $(r_i,d_i)$, $i\ge1$ and positive constants $Q,Q'$ such that
	\begin{enumerate}
		\item $r_1=r>0$, 
		\item $r_{i+1}=r_i+d_i$,
		\item $\frac{r_i}{d_i}\in \mathbb N$,
		\item $Qr_i^q\le d_i\le Q'r_i^q$,
		\item $r_i\to\infty$.
	\end{enumerate}

\begin{proof} Let us assume that we are given a sequence of positive integers $(n_i)$, $i\ge1$ and that we define $r_i$ and $d_i$ recursively by 
	\begin{align*}
		r_1&=r\,,\\
		d_i&=\frac{r_i}{n_i}\,,\\
		r_{i+1}&=r_i+d_i\,.
	\end{align*}
	Then, properties (1)--(3) are automatically satisfied. It remains to choose $(n_i)$ in such a way that (4)--(5) hold. Dividing by $r_i^q$, (4) can be rewritten
	\begin{equation}\label{ineq}
		Q\le \frac{r_{i}^{1-q}}{n_i}\le Q'\,.
	\end{equation}
	We can define $n_i$ by 
	\begin{equation*}
		n_i=\left\lceil r_i^{1-q}\right\rceil\,,
	\end{equation*}
	where $\lceil x \rceil$ is the smallest integer at least as large as $x$. Note that by combining this definition with the previous one, we construct the sequences $(r_i,d_i)$ and $(n_i)$ recursively, starting from $r_1=r$. The sequence $(r_i)$ is increasing, so it either is bounded or goes to $\infty$. If it was bounded, then we would have $n_i\le N$ for some integer $N$, and therefore 
	\begin{equation*}
		r_{i+1}\ge \left(1+\frac1N\right)r_i\,.
	\end{equation*}
	This implies $r_i\to\infty$, in contradiction with the assumption. Therefore $r_i\to\infty$, so (5) holds. Finally, we have
	\begin{equation*}
		\ r_i^{1-q}\le n_i<r_i^{1-q}+1\,.
	\end{equation*}
	We deduce
	\begin{equation*}
		\frac1{1+\ r^{q-1}}\le\frac1{1+\ r_i^{q-1}}\le\frac{r_{i}^{1-q}}{n_i}\le 1\,.
	\end{equation*}
	Since $d_i=r_i/n_i$, this shows that (4) holds with $Q=1/(1+r^{q-1})$ and $Q'=1$\,.
\end{proof}

\begin{remark}\label{rmrkQ} For any sequence $(r_i,d_i)$ satisfying the properties of lemma \ref{lemSeq}, there exists positive constants $Q'',Q'''$ such that 
	\begin{equation}\label{eqPart}
		Q''\le\frac{d_{i+1}}{d_i}\le Q'''\,.
	\end{equation}
\end{remark}
Indeed, dividing property (4) by $r_i $, we find 
\begin{equation*}
	Qr_i^{q-1}\le \frac{d_{i}}{r_i}\le Q'r_i^{q-1}\,,
\end{equation*} 
and property (5) then implies $d_i/r_i\to0$. We then get, from property (2),
\begin{equation*}
	\frac{r_{i+1}}{r_i}=1+\frac{d_i}{r_i}\to 1\,.
\end{equation*}
Since we have 
\begin{equation*}
	\frac{Q}{Q'}\left(\frac{r_{i+1}}{r_i}\right)^q\le\frac{d_{i+1}}{d_i}\le\frac{Q'}{Q}\left(\frac{r_{i+1}}{r_i}\right)^q
\end{equation*}
from property (4), the desired result follows.

\bigskip

\section{Declaration}

The authors have no conflicts of interest to declare that are relevant to the content of this article.

\bibliographystyle{alpha}


\end{document}